\documentclass[preprint, 12pt, english]{elsarticle}
\usepackage{amsmath}
\usepackage{latexsym, amssymb}
\usepackage{amsthm}
\usepackage{txfonts}

\newtheorem{thm}{Theorem}[section] 

\newtheorem{cor}[thm]{Corollary}

\newtheorem{lem}[thm]{Lemma}
\newtheorem{prop}[thm]{Proposition}

\newcommand\operA[2]{{\if!#2!\operatorname{#1}\else{\operatorname{#1}_{#2}^{\phantom{I}}}\fi}} 
\newcommand\set[1]{\{#1\}}

\newcommand\Cref[1]{{Corollary~\ref{#1}}}%

\def\norm{{\operatorname{N}}}

\newcommand{\Trace}[1][]{\if!#1!\operatorname{Tr}\else{\operatorname{Tr}_{#1}^{\phantom{I}}}\fi} 

\long\def\forget#1\forgotten{{}} %

\def\({\left(}
\def\){\right)}

\newif\iffurther
\furtherfalse

\newif\ifXY 
\XYtrue     
\ifXY

\input xy
\input xyidioms.tex
\usepackage{xy}
\xyoption{all} %
\fi 

\usepackage{babel}

\journal{??}

\begin{document}

\begin{frontmatter}

\title{Kummer Subspaces of Tensor Products of Cyclic Algebras}

\author{Adam Chapman\corref{ch}}
\ead{adam1chapman@yahoo.com}
\cortext[ch]{The author is supported by Wallonie-Bruxelles International.}
\address{ICTEAM Institute, Universit\'{e} Catholique de Louvain, B-1348 Louvain-la-Neuve, Belgium.}

\begin{abstract}
We discuss the Kummer subspaces of tensor products of cyclic algebras, focusing mainly on the case of cyclic algebras of degree 3.
We present a family of maximal spaces in the general case, classify all the monomial spaces in the case of tensor products of cyclic algebras of degree 3 using graph theory, and provide an upper bound for the dimension in the generic tensor product of cyclic algebras of degree 3.
\end{abstract}


\end{frontmatter}

\section{Introduction}

Let $p$ be a prime number and $F$ be an infinite field of characteristic not $p$ containing a primitive $p$th root of unity $\rho$.
A cyclic algebra of degree $p$ over $F$ is an algebra that can be presented as
$$F[x,y : x^p=\alpha, y^p=\beta, y x=\rho x y]$$
for some $\alpha,\beta \in F^\times$.
We denote the presentation as $(\alpha,\beta)_{p,F}$.
A given algebra can have more than one presentation.
Fixing a presentation, we call the elements $x^i y^j$ where $i$ and $j$ are integers between $0$ and $p-1$ ``monomials".
The same goes for tensor products of cyclic algebras, i.e. if we fix presentations $F[x_k,y_k: x_k^p=\alpha_k, y_k^p=\beta_k, y_k x_k=\rho x_k y_k]=(\alpha_k,\beta_k)_{p,F}$ then the monomials in the tensor product $\bigotimes_{k=1}^n (\alpha_k,\beta_k)_{p,F}$ are $\prod_{k=1}^n x_k^{i_k} y_k^{j_k}$.

Let $A$ be a division tensor product of $n$ cyclic algebras of degree $p$. An element $x \in A$ is called Kummer (or $p$-central) if $x^p \in F$ whereas $x^k \not \in F$ for every $1 \leq k \leq p-1$.
An $F$-vector subspace of $A$ is called Kummer if all its nonzero elements are Kummer.
A necessary and sufficient condition for a space $F x_1+\dots+F x_m$ to be Kummer is that $x_1^{d_1} * \dots *x_m^{d_m} \in F$ for every $m$-tuple of nonnegative integers $d_1,\dots,d_m$ satisfying $d_1+\dots+d_m=p$. The expression $x_1^{d_1} * \dots *x_m^{d_m}$ stands for the sum of all the words in which each $x_k$ appears $d_k$ times, e.g. $x_1^2 * x_2=x_1^2 x_2+x_1 x_2 x_1+x_2 x_1^2$. This notation was introduced in \cite{Revoy}.
Fixing the presentations of the cyclic algebras, a Kummer space is called ``monomial" if it is spanned by monomials.

The classification of Kummer spaces is an open problem. There has not even been found yet an upper bound for the dimension of such spaces, except for a few special cases. If $p=2$ and $n=1$ the space of elements of trace zero contains all the Kummer elements. In general it is known that for $p=2$ the size of the Kummer space is bounded from above by $2 n+1$.
The Kummer spaces in case of $p=n=2$ were studied in more detail in \cite{ChapVish2}.
The Kummer spaces in case of $p=3$ and $n=1$ where classified in \cite{Raczek} and \cite{MV1}.
So far the formula $p n+1$ for the upper bound of the dimension holds in all the known cases, and we conjecture it to be true in general.

In \cite{CGMRV} the monomial Kummer spaces in division cyclic algebras of prime degrees were classified.
Furthermore, an upper bound was provided for the dimension of the Kummer subspaces of the generic cyclic algebra, i.e. the algebra $(\alpha,\beta)_{p,K}$ where $K$ is the purely transcendental field extension of $F$ generated by $\alpha$ and $\beta$.

In Section \ref{maxSec} we present a family of maximal Kummer subspaces (with respect to inclusion) of tensor products of cyclic algebras of any degree. These spaces happen also to be monomial. This section is taken from \cite{ChapmanPHD} and is based in turn on a result from \cite{Chapmanthesis}.

In Section \ref{threesec} we classify the monomial Kummer subspaces of tensor products of cyclic algebras of degree 3.
For their description we make use of graph theory.
This section is based on results from \cite{ChapmanPHD}.

In Section \ref{generic} we provide an upper bound for the dimension of a Kummer subspace of the generic tensor product of cyclic algebras of degree 3.

\section{Maximal Kummer Subspaces of Tensor Products of Cyclic Algebras}\label{maxSec}

Fix $A=\bigotimes_{k=1}^n (\alpha_k,\beta_k)_{p,F}=\bigotimes_{k=1}^n F[x_k,y_k : x_k^p=\alpha_k, y_k^p=\beta_k, y_k x_k=\rho x_k y_k]$.

Let $V_0=F$ and $V_k=F[x_k] y_k+V_{k-1} x_k^{a_k}$ for any $1 < k \leq n$ and $a_k$ prime to $p$.
Assume that $v^p \in F$ for all $v \in V_{k-1}$ for a certain $k$.
Every element of $V_k$ is of the form $f(x_k) y_k+v x_k^{a_k}$ for some $f(x_1) \in F[x_1]$ and $v \in V_{k-1}$. Since $v$ commutes with $x_k$ and $y_k$, and $y_1 x_1=\rho x_1 y_1$, $(f(x_k) y_k+v x_k^{a_k})^p=(f(x_k) y_k)^p+v^p x_k^{p a_k}=\norm_{F[x_k]/F}(f(x_k)) \beta_k+v^p \alpha_k^{a_k} \in F$.
For any $1 \leq m \leq p-1$, if $f(x_k) \neq 0$ then the eigenvector of $(f(x_k) y_k+v x_k^{a_k})^m$ corresponding to the eigenvalue $\rho^m$ with respect to conjugation by $x$ is $(f(x_k) y_k)^m$, which is not zero, and therefore $(f(x_k) y_k+v x_k^{a_k})^m \not \in F$.
If $f(x_k)=0$ then what is left is $v x_k^{a_k}$, and of course $v^m x_1^{a_k m} \not \in F$.
Consequently, $V_k$ is Kummer.
Since $V_0=F$, by induction $V_k$ is Kummer for every $1 \leq k \leq n$. The dimension of each $V_k$ is $p k+1$.

\begin{thm}
For any $k \leq n$, $V_k$ is maximal with respect to inclusion.
\end{thm}

\begin{proof}
Let $V=V_k$. $V$ has a standard basis $$B=\set{x_i^j y_i x_{i+1} \dots x_k : 1 \leq i \leq k, 0 \leq j \leq p-1} \cup \set{x_1 x_2 \dots x_k}.$$
Let $z$ be a nonzero element in the algebra $A$.
This element can be expressed as a linear combination of the monomials $x_1^{c_1} y_1^{e_1}\dots x_n^{c_n} y_n^{e_n}$.
Let us assume negatively that $V+F z$ is $p$-central.
Consequently, $w^{p-1} * z \in F$ for every $w \in B$.
Since we can subtract from $z$ the appropriate linear combination of the elements of $B$, we can assume that $w^{p-1} * z=0$ for every $w \in B$.

Let us pick one monomial $t=x_1^{c_1} y_1^{e_1}\dots x_n^{c_n} y_n^{e_n}$.

If $e_1=e_2=\dots=e_n=0$ then $t$ commutes with $x_1 x_2 \dots x_k \in V$. Since $(x_1 x_2 \dots x_k)^{p-1} * z=0$, the coefficient of $t$ in $z$ is zero.

Otherwise, let $i$ be the maximal integer for which $e_i \neq 0$.
The monomial $t$ commutes with the element $x_i^r y_i x_{i+1} \dots x_k \in V$ where $r \equiv c_i e_i^{-1} \pmod{p}$.
Since $(x_i^r y_i x_{i+1} \dots x_k)^{p-1} * z=0$, the coefficient of $t$ in $z$ is zero.

Since it holds for any monomial $t$, $z=0$, and that is a contradiction.
\end{proof}

\section{Monomial Kummer subspaces of Tensor Products of Cyclic Algebras of Degree 3}\label{threesec}

Keep $A$ as before and assume $p=3$.
Let $\mathcal{X}$ be the set of all Kummer elements in $A$.
We build a directed graph $(\mathcal{X},E)$ by drawing an edge from $y$ to $x$
$$\xymatrix@C=20px@R=20px{y \ar@{->}[r]  & x}$$ if $y x y^{-1}=\rho x$.
For any subset $B \subset \mathcal{X}$, $(B,E_B)$ is the subgraph obtained by taking the vertices in $B$ and all the edges between them.
The set $B$ is called $\rho$-commuting if it is linearly independent over $F$ and for every two distinct elements $x,y \in B$, $y x y^{-1}=\rho^k k$ for $k \in \set{0,1,2}$.
In particular, any set of monomials is $\rho$-commuting.

According to \cite[Corollary 2.2]{ChapVish1}, a set $\set{x_1,\dots,x_m}$ spans a Kummer space if and only if every subset of cardinality three $\set{x_i,x_j,x_k}$ spans a Kummer space.
Therefore we will start with the set of cardinality 3.

\begin{lem}\label{size3}
Given a $\rho$-commuting set $\set{x,y,z}$,  $F x+F y+F z$ is Kummer if and only if either
\begin{equation*}
\xymatrix@R=12pt@C=18pt{
{x}\ar@{->}|(0.5){}[r]\ar@{->}|(0.5){}[d] & {y}\\
{z}\ar@{->}|(0.5){}[ru] & {}
}
\end{equation*}
or $x y z \in F$, in which case
\begin{equation*}
\xymatrix@R=12pt@C=18pt{
{x}\ar@{<-}|(0.5){}[r]\ar@{->}|(0.5){}[d] & {y}\\
{z}\ar@{->}|(0.5){}[ru] & {}
}
\end{equation*}
\end{lem}

\begin{proof}
If $x y=y x$ then $x^2*y=3 x^2 y \in F$ which means that $y \in F x$, contradiction.
Consequently we are left with the two graphs above (up to a change in the order of the elements).
In the first case, $x * y * z=0$, so there are no extra conditions.
In the second case, $x * y * z=-3 \rho^{-1} x y z \in F$.
The opposite direction is a straight-forward computation.
\end{proof}

Let $B$ be a $\rho$-commuting set spanning a Kummer space. We will now study the properties of the directed graph $(B,E_B)$.
By a cycle we always mean a \textbf{simple directed cycle}.

\begin{prop}\label{direction}
If $(B,E_B)$ contains a cycle of length $3$
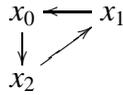
\begin{equation*}
\xymatrix@R=12pt@C=18pt{
{x_0}\ar@{<-}|(0.5){}[r]\ar@{->}|(0.5){}[d] & {x_1}\\
{x_2}\ar@{->}|(0.5){}[ru] & {}
}
\end{equation*}
then for every $y \in B \setminus \set{x_0,x_1,x_2}$, either $\xymatrix@C=20px@R=20px{x_k \ar@{->}[r]  & y}$ for any $k \in \set{0,1,2}$ or $\xymatrix@C=20px@R=20px{x_k \ar@{<-}[r]  & y}$ for any $k \in \set{0,1,2}$.
\end{prop}

\begin{proof}
If $\xymatrix@C=20px@R=20px{x_0 \ar@{->}[r]  & y}$ and $\xymatrix@C=20px@R=20px{x_1 \ar@{<-}[r]  & y}$ then
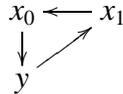
\begin{equation*}
\xymatrix@R=12pt@C=18pt{
{x_0}\ar@{<-}|(0.5){}[r]\ar@{->}|(0.5){}[d] & {x_1}\\
{y}\ar@{->}|(0.5){}[ru] & {}
}
\end{equation*}
which means that $y x_0 x_1 \in F$. Since $x_0 x_1 x_2 \in F$, we get $y \in F x_2$, which contradicts the linear independence.
The rest of the proof repeats the same idea.
\end{proof}

\begin{prop}
The cycles of $(B,E_B)$ are vertex-disjoint.
\end{prop}

\begin{proof}
First assume that
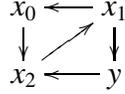
\begin{equation*}
\xymatrix@R=12pt@C=18pt{
{x_0}\ar@{<-}|(0.5){}[r]\ar@{->}|(0.5){}[d] & {x_1}\ar@{->}|(0.5){}[d]\\
{x_2}\ar@{->}|(0.5){}[ru] &  y\ar@{->}|(0.5){}[l] & {}
}
\end{equation*}
Then $y x_1 x_2 \in F$ whereas $x_0 x_1 x_2 \in F$, which means that $y \in F x_0$, and that contradicts the linear independence.

Assume that
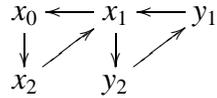
\begin{equation*}
\xymatrix@R=12pt@C=18pt{
{x_0}\ar@{<-}|(0.5){}[r]\ar@{->}|(0.5){}[d] & {x_1}\ar@{<-}|(0.5){}[r]\ar@{->}|(0.5){}[d] & y_1\\
{x_2}\ar@{->}|(0.5){}[ru] &  y_2\ar@{->}|(0.5){}[ru] & {}
}
\end{equation*}
From Proposition \ref{direction} we have $\xymatrix@C=20px@R=20px{x_0 \ar@{->}[r]  & y_2}$ and $\xymatrix@C=20px@R=20px{y_1 \ar@{->}[r]  & x_0}$.
But then
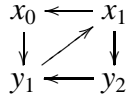
\begin{equation*}
\xymatrix@R=12pt@C=18pt{
{x_0}\ar@{<-}|(0.5){}[r]\ar@{->}|(0.5){}[d] & {x_1}\ar@{->}|(0.5){}[d]\\
{y_1}\ar@{->}|(0.5){}[ru] &  y_2\ar@{->}|(0.5){}[l] & {}
}
\end{equation*}
and we saw already that this is impossible.
\end{proof}

\begin{prop}
There are no cycles of length greater than 3.
\end{prop}

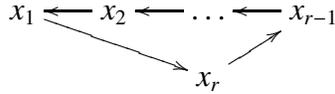
\begin{proof}
Assume
\begin{equation*}
\xymatrix@R=12pt@C=18pt{
{x_1}\ar@{<-}|(0.5){}[r]\ar@{->}|(0.5){}[drr] & {x_2}\ar@{<-}|(0.5){}[r] & \dots\ar@{<-}|(0.5){}[r]& x_{r-1}\\
& &{x_r}\ar@{->}|(0.5){}[ru] & {}
}
\end{equation*}
for some $r \geq 4$.
Let $i$ be the maximal integer between $1$ and $r$ such that $\xymatrix@C=20px@R=20px{x_i \ar@{->}[r]  & x_1}$.
Now, $\xymatrix@C=20px@R=20px{x_1 \ar@{->}[r]  & x_{i+1}}$.
Therefore
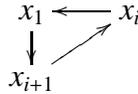
\begin{equation*}
\xymatrix@R=12pt@C=18pt{
{x_1}\ar@{<-}|(0.5){}[r]\ar@{->}|(0.5){}[d] & {x_i}\\
{x_{i+1}}\ar@{->}|(0.5){}[ru] & {}
}
\end{equation*}
If $i \geq 3$ then according to Proposition \ref{direction}, $\xymatrix@C=20px@R=20px{x_1 \ar@{->}[r]  & x_{i-1}}$, which implies that $i \neq 3$, or in other words $i \geq 4$. Let $j$ be the minimal index for which $\xymatrix@C=20px@R=20px{x_1 \ar@{->}[r]  & x_{j+1}}$. In particular $\xymatrix@C=20px@R=20px{x_j \ar@{->}[r]  & x_1}$. Now, $j+1 \leq i-1$, which means that
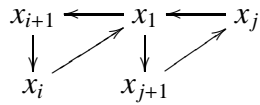
\begin{equation*}
\xymatrix@R=12pt@C=18pt{
{x_{i+1}}\ar@{<-}|(0.5){}[r]\ar@{->}|(0.5){}[d] & {x_1}\ar@{<-}|(0.5){}[r]\ar@{->}|(0.5){}[d] & x_j\\
{x_i}\ar@{->}|(0.5){}[ru] &  x_{j+1}\ar@{->}|(0.5){}[ru] & {}
}
\end{equation*}
But this is impossible.
If $i=2$ then according to Proposition \ref{direction}, $\xymatrix@C=20px@R=20px{x_4 \ar@{->}[r]  & x_1}$ which contradicts the maximality of $i$.
\end{proof}

As a consequence we obtain the following theorem:
\begin{thm}
A $\rho$-commuting subset $B$ of $\mathcal{X}$ spans a Kummer space if and only if the graph $(B,E_B)$ satisfies the following axioms:
\begin{enumerate}
\item For every two distinct elements $x,y \in B$, either $\xymatrix@C=20px@R=20px{ x\ar@{->}[r]  & y}$ or $\xymatrix@C=20px@R=20px{ x\ar@{<-}[r]  & y}$
\item All cycles are of length 3.
\item The product of all the elements in a cycle is in $F$.
\item The cycles are vertex-disjoint.
\end{enumerate}
\end{thm}

\begin{proof}
The straight-forward direction is an immediate result of what we did so far.
The opposite direction is a result of the fact that every three elements in this set span a Kummer space according to Lemma \ref{size3}.
\end{proof}

\begin{cor}
Given a $\rho$-commuting set $B$ spanning a Kummer space, if $\#B=m$ then the longest path
$\xymatrix@C=20px@R=20px{ x_1\ar@{->}[r]  & x_2\ar@{->}[r] & \dots\ar@{->}[r] & x_r}$ in the graph $(B,E_B)$ satisfying $\xymatrix@C=20px@R=20px{ x_i\ar@{->}[r]  & x_j}$ for any $1 \leq i <j \leq r$ is of length no less than $m-\lfloor \frac{m}{3} \rfloor$.
\end{cor}

\begin{proof}
Take $B$ and take off exactly one element from each cycle.
The number of elements taken off is at most $\lfloor \frac{m}{3} \rfloor$, and what is left satisfies the required condition.
\end{proof}

\begin{cor}
The maximal $\rho$-commuting set spanning a Kummer space in $A$ is of cardinality $3 n+1$.
\end{cor}

\begin{proof}
We are already familiar with monomial Kummer spaces of size $3 n+1$.
According to the previous corollary, if we have a $\rho$-commuting set $B$ of size $3 n+2$ spanning a Kummer space then we have a path in $(B,E_B)$
$$\xymatrix@C=20px@R=20px{ x_1\ar@{->}[r]  & x_2\ar@{->}[r] & \dots\ar@{->}[r] & x_{2 n+2}}$$ satisfying $\xymatrix@C=20px@R=20px{ x_i\ar@{->}[r]  & x_j}$
for any $1 \leq i <j \leq 2 n+2$. Then the set $B$ generates over $F$ a tensor product of $n+1$ cyclic algebras of degree 3
$$F[x_1,x_2] \otimes F[x_1 x_2^{-1} x_3,x_1 x_2^{-1} x_4] \otimes \dots \otimes F[(\prod_{k=1}^n x_{2 k-1} x_{2 k}^{-1}) x_{2 n+1},(\prod_{k=1}^n x_{2 k-1} x_{2 k}^{-1}) x_{2 n+2}],$$
contradiction.
\end{proof}

\section{The Generic Tensor Product of Cyclic Algebras}\label{generic}

Let $K$ be a purely transcendental field extension of $F$ generated by $\set{\alpha_k,\beta_k : 1 \leq k \leq n}$.
Fix $A=\bigotimes_{k=1}^n K[x_k,y_k : x_k^p=\alpha_k, y_k^p=\beta_k, y_k x_k=\rho x_k y_k]$.

\begin{thm}
For any Kummer subspace of $A$ there exists a monomial Kummer space of the same dimension.
\end{thm}

\begin{proof}
Write $V=F v_1+F v_2+\dots+F v_m$, where $n$ is the dimension of $V$.
Each $v_k$ is a sum of monomials of the form $c x_1^{a_1} y_1^{b_1} \dots x_n^{a_n} y_n^{b_n}$ where the coefficient $c$ is a quotient of two polynomials in the variables $\alpha_1,\beta_1,\dots,\alpha_n,\beta_n$ over $F$.
By multiplying by all the denominators, we can assume that the coefficients are polynomials.
Then each $v_k$ can be written also as a polynomial in the variables $x_1,y_1,\dots,x_n,y_n$ over $F$.
We now impose a lexicographical valuation on the polynomials in $F[x_1,y_1,\dots,x_n,y_n]$.
In particular, every polynomial has now a leading monomial, i.e. the monomial $x_1^{a_1} y_1^{b_1} \dots x_n^{a_n} y_n^{b_n}$ of the highest value with a nonzero coefficient.

By the following process we can make sure the leading monomials of $v_1,\dots,v_m$ are distinct and linearly independent over $K$:
If the leading monomial of $v_1$ is $x_1^{a_1} y_1^{b_1} \dots x_n^{a_n} y_n^{b_n}$ then we take the coefficient $c$ of $x_1^{a_1'} y_1^{b_1'} \dots x_n^{a_n'} y_n^{b_n'}$ in $v_1$ when writing $v_1$ as a polynomial in $(F[\alpha_1,\beta_1,\dots,\alpha_n,\beta_n])[x_1,y_1,\dots,x_n,y_n]$ where $a_k',b_k'$ are the unique integers between 0 and $p-1$ satisfying $a_k' \equiv a_k ,b_k' \equiv b_k \pmod{p}$. Then for each $2 \leq i \leq m$ we replace each $v_i$ with $c v_i-c_i v_1$ where $c_i$ is the coefficient of that monomial in $v_i$.
Then we fix $v_2$ and change $v_3,\dots,v_m$ similarly, and so on.

Let $w_k$ be the leading monomial of $v_k$.
For any set of nonnegative integers $d_1,\dots,d_m$ satisfying $d_1+\dots+d_m=p$, The expression $w_1^{d_1} * \dots * w_m^{d_m}$ is either equal to the leading monomial of $v_1^{d_1} *\dots* v_m^{d_m}$ or to zero. In both cases, it is in $F$, which means that $F w_1+\dots+F w_m$ is monomial Kummer.
\end{proof}

\begin{cor}
If $p=3$ then the upper bound for the dimension of a Kummer subspace of $A$ is $3 n+1$.
\end{cor}

\begin{proof}
Follows immediately from the previous theorem and Section \ref{threesec}.
\end{proof}

\section*{Acknowledgements}
I owe thanks to Jean-Pierre Tignol and Uzi Vishne for their help and support.

\section*{Bibliography}
\bibliographystyle{amsalpha}
\bibliography{bibfile}
\end{document}